\documentclass[a4paper,12pt,oneside]{article}

\usepackage[
    a4paper,
    top=28mm,
    bottom=28mm,
    left=20mm,
    right=17mm
]{geometry}

\usepackage[T1]{fontenc}
\usepackage{mathtools}
\usepackage{amssymb}
\usepackage{amsthm}
\usepackage{mathrsfs}
\usepackage{latexsym}

\usepackage{graphicx}
\usepackage{epsfig}
\usepackage{tikz}
\usepackage[all]{xy}
\usepackage{amscd}
\usepackage{booktabs}

\usepackage{authblk}
\usepackage{xifthen}
\usepackage{xcolor}
\usepackage{verbatim}
\usepackage{etoolbox}

\usepackage{hyperref}

\hypersetup{
    colorlinks=true,
    linkcolor=orange,
    citecolor=magenta,
    urlcolor=blue
}

\theoremstyle{plain}

\newtheorem{theorem}{Theorem}[section]
\newtheorem{lem}[theorem]{Lemma}
\newtheorem{prop}[theorem]{Proposition}
\newtheorem{cor}[theorem]{Corollary}
\newtheorem{conj}[theorem]{Conjecture}

\newtheorem{theoremx}{Theorem}

\newtheorem{corx}[theoremx]{Corollary}

\theoremstyle{definition}

\newtheorem{defn}[theorem]{Definition}
\newtheorem{rem}[theorem]{Remark}

\numberwithin{equation}{section}

\makeatletter
\patchcmd{\@startsection}
  {\@afterindentfalse}
  {\@afterindentfalse\@afterheading}
  {}{}
\makeatother

\title{On the classification of perfect Prishchepov groups}

\author[1]{Layla Sorkatti%
    \thanks{Email: \texttt{layla.sorkatti@gmail.com}}}

\author[2]{Ihechukwu Chinyere%
    \thanks{Corresponding author: \texttt{i.chinyere@up.ac.za; ihechukwu@aims.ac.za}}}

\affil[1]{Department of Pure Mathematics, University of Khartoum and Al Neelain University, Khartoum, Sudan.}

\affil[2]{Department of Mathematics and Applied Mathematics, University of Pretoria, Hatfield 0028, Pretoria, South Africa.}

\begin{document}
\maketitle

\begin{abstract}
The Prishchepov groups $P(r,n,k,s,q)$ form a broad class of cyclically presented groups. We verify a conjectural characterisation of the perfect groups in this family. We first prove the conjecture for the case $\gcd(n,6)=1$ and then establish further cases beyond this coprimality condition. Consequently, we obtain a classification of perfect Prishchepov groups in a broad range of parameters.

\vspace{1em}

\end{abstract}

\section{Introduction}

Let \(n \geq 1\) be a natural number, and let \(F_n\) denote the free group of rank \(n\) with generators \(x_0, x_1, \ldots, x_{n-1}\). Define the automorphism \(\theta \in \operatorname{Aut}(F_n)\) by
\[
\theta(x_i) = x_{(i+1) \bmod n}, \quad 0 \leq i < n.
\]
A \emph{cyclic presentation} is a group presentation of the form
\[
P_n(w) = \langle x_0, x_1, \ldots, x_{n-1} \mid w, \theta(w), \ldots, \theta^{n-1}(w) \rangle,
\]
where \(w \in F_n\) is the defining relator. The group defined by the presentation \(P_n(w)\) is denoted \(G_n(w)\). A group is said to be \emph{cyclically presented} if it is isomorphic to \(G_n(w)\) for some word \(w\) and natural number \(n\).

\medskip

Among the many families of cyclically presented groups, one of the most influential is the family introduced by Prishchepov \cite{prishchepov1995aspherisity}, denoted \(P(r,n,k,s,q)\), where \(n,k,r,s,q \ge 1\) are integers. This family admits a cyclic presentation \(P_n(w)\) with defining word
\[
w = \left(\prod_{i=0}^{r-1} x_{iq}\right) \left(\prod_{i=0}^{s-1} x_{k-1+iq}\right)^{-1},
\]
subscripts being read modulo $n$. Prishchepov's formulation unifies and extends several earlier constructions. For example, the following subfamilies arise as special cases:

\begin{itemize}
    \item the Fibonacci groups \(F(2,n) = P(2,n,3,1,1)\) \cite{conway1965advanced},
    \item the family \(G_n(m,k) = P(2,n,k+1,1,m)\) \cite{MR1634446,johnson1975some},
    \item the family \(H(r,n,s) = P(r,n,r+1,s,1)\) \cite{campbell1975class},
    \item the Gilbert--Howie groups \(H(n,m) = P(2,n,2,1,m)\) \cite{GHow}, and
    \item the Sieradski groups \(S(r,n) = P(r,n,2,r-1,2)\) \cite{MR830041,cavicchioli1998geometric}.
\end{itemize}
Consequently, the Prishchepov groups provide a unifying framework for studying many families of cyclically presented groups, while simultaneously extending and generalising results previously obtained for those families.

\medskip
A central theme in the theory of cyclically presented groups is the study of algebraic properties of \(G_n(w)\) as \(n\) varies. One property of particular interest is \emph{perfectness}, that is, the vanishing of the abelianisation. This problem has been investigated for many subfamilies, including \(H(r,n,s)\) \cite{MR4315549,MR4418964,williams2019generalized}, \(H(n,m)\) \cite{Odoni,GHow}, and \(G_n(m,k)\) \cite{MR2488144,Will1}, as well as for the Prishchepov groups \cite{MR4315549}. Other work has focused on structural properties such as the finiteness of abelianisations \cite{MR1967243}, within larger families of cyclically presented groups in which the Prishchepov family appears as a subfamily.

\medskip
Perfectness is almost obvious in certain cases, leading to the following definition. Throughout we write
\[
d := \gcd(n,q).
\]

\begin{defn}[Trivially perfect]\label{def:tperf}
The group \(P(r,n,k,s,q)\) is \emph{trivially perfect} if $|r-s|=1$, $d \mid (k-1)$, and at least one of the following holds modulo \(n\):
\[
k \equiv 1,\quad
k \equiv 1+q,\quad
d\,r \equiv 0,\quad
d\,s \equiv 0 .
\]
\end{defn}

\begin{rem}\label{rem:defnfix}
The divisibility hypothesis \(d\mid k-1\) is essential and cannot be omitted. It is automatic in the first two cases, since \(k\equiv1\) and \(k\equiv1+q\) give \(k-1\equiv0\) and \(k-1\equiv q\) modulo \(n\), respectively, and \(d\mid n\), \(d\mid q\). It is \emph{not} automatic in the last two cases. For example, \(P(2,4,2,1,2)\) satisfies \(|r-s|=1\) and \(dr = 2\cdot 2\equiv0 \pmod 4\), but its associated polynomial is \(f(t)=1-t+t^{2}\) and
\[
\prod_{\zeta^{4}=1} f(\zeta)= f(1)f(-1)f(i)f(-i) = 1\cdot 3\cdot(-i)\cdot i = 3,
\]
so that \(P(2,4,2,1,2)^{\mathrm{ab}}\cong\mathbb Z_3\) and the group is not perfect. Similarly, \(P(5,25,2,4,5)\) satisfies \(dr = 5\cdot 5\equiv 0\pmod{25}\); on the primitive fifth roots of unity, its associated polynomial takes the value \(5-4\zeta\), contributing the factor \((5^{5}-4^{5})/(5-4)=2101\) to the resultant, while the remaining factors are roots of unity, so the resultant is \(\pm2101\) and this group is also not perfect. Note that \(d\nmid k-1\) in both examples. By contrast, perfectness always forces \(d\mid k-1\) (see \cite[Lemma~2.3]{MR4315549}), so imposing this condition does not cost anything in the statements below.
\end{rem}

That trivially perfect groups are indeed perfect follows from \cite[Theorem 17(b)]{MR2946300} together with Proposition~\ref{tperfect}.

\medskip
Closely related to perfectness is the question of triviality. Since every trivial group is necessarily perfect, characterising perfect cyclically presented groups provides a natural first step toward identifying trivial ones. The literature on cyclic presentations of the trivial group is extensive, spanning combinatorial, geometric, and algebraic approaches; see, for instance, \cite{MR1837678,MR3223773,MR2043998,MR4365040,prishchepov1995aspherisity,williams2019generalized,MR3010816}. In particular, Problem 1.1 of \cite{MR3223773} asks which cyclically presented groups are trivial.

\medskip
To make progress with the analysis of the perfect Prishchepov groups, it is convenient to isolate a subclass defined by certain congruence conditions.

\begin{defn}[type $\widetilde{\mathfrak{Z}}$, cf.~\cite{MR4315549}]\label{def:typeZtilde}
The group \(P(r,n,k,s,q)\) is said to be of \emph{type \(\widetilde{\mathfrak{Z}}\)} if one of the following congruences holds modulo \(n\):
\begin{enumerate}
    \item \textbf{type \(\mathfrak{Z}\):}
    \(\; q(r-s) \equiv 2(k-1) \pmod{n};\)
    \item \textbf{type \(\mathfrak{Z}'\):}
    \(\; q(r+s) \equiv 0 \pmod{n}.\)
\end{enumerate}
\end{defn}

The type \(\mathfrak{Z}\) condition was first identified in \cite{MR4365040} in connection with topological applications, though it has implicitly appeared in earlier work. In general, for groups of type \(\widetilde{\mathfrak{Z}}\), perfectness is characterised by the following result.

\begin{theorem}[{\cite[Theorem C]{MR4315549}}]\label{ThmC}
Let \(n \geq 2,\; k,q \geq 1\), and \(r \geq s \geq 1\) with \(\gcd(n,k-1,q)=1\).
If \(P(r,n,k,s,q)\) is of type \(\widetilde{\mathfrak{Z}}\), then the following are equivalent:
\begin{enumerate}
\item \(P(r,n,k,s,q)\) is perfect;
\item \(|r-s|=1,\; \gcd(n,q)=1,\) and \(\gcd(k-1-qr,n)=1.\)
\end{enumerate}
\end{theorem}

\medskip
Connections with low-dimensional topology further emphasise the significance of these groups. The Brieskorn manifold \(M(a,b,c)\) is homeomorphic to the \(c\)-fold cyclic branched covering of \(S^3\) along the torus link of type \((a,b)\), and for the Sieradski groups \(S(r,n) = P(r,n,2,r-1,2)\) one has
\[
S(r,n) \cong \pi_1(M(2,2r-1,n));
\]
see \cite{MR1634446,MR418127}.
By Milnor's classification, the spherical cases occur precisely for \((a,b,c) = (2,3,3)\), \((2,3,4)\), and \((2,3,5)\). In these cases, the fundamental group \(\pi_1(M(a,b,c))\) is finite but nontrivial, namely the binary tetrahedral, binary octahedral, and binary icosahedral groups, of orders \(24\), \(48\), and \(120\), respectively. Exploiting this connection, \cite{MR4315549} established the following description of trivial groups of type \(\mathfrak{Z}\), which essentially states that the trivial groups in this class are the obvious ones.

\begin{cor}[{\cite[Corollary D]{MR4315549}}]\label{corD}
Let \(n \geq 2,\; k,q \geq 1,\; r \geq s \geq 1\) with \(\gcd(n,k-1,q)=1\).
If \(P(r,n,k,s,q)\) is of type \(\mathfrak{Z}\), then the following are equivalent:
\begin{enumerate}
\item \(P(r,n,k,s,q)\) is trivial;
\item \(|r-s|=1\) and either \(k \equiv 1 \pmod{n}\) or \(k \equiv 1+q \pmod{n}\).
\end{enumerate}
\end{cor}

\begin{rem}
We mention an overlooked application of the existing theory in \cite{MR4315549}.
Both examples below arise from the group
\[
E_n(t)=\langle x,t\mid t^n=1,\;w(x,t)=1\rangle,
\]
with $n=3$, the only difference being the choice of the word $w(x,t)$. For the group labelled (O100) in \cite{EdjvetSwan_IrreducibleCyclicallyPresented},
we have
\[
w(x,t)=xtxtxtxt^{-2}x^{-1}t^2x^{-1}t^{-1}x^{-1}t^{-2}.
\]
Rewriting the presentation in terms of $x_i=t^ixt^{-i}$ gives $P(4,3,3,3,1)$. Here $\gcd(n,k-1,q)=\gcd(3,2,1)=1$, and
\[
q(r-s)=1\equiv 4 = 2(k-1) \pmod 3,
\]
so the group is of type $\mathfrak{Z}$. Since $k=3\not\equiv 1$ and $k\not\equiv 1+q=2 \pmod 3$, Corollary~\ref{corD} shows that it is nontrivial.

\medskip
For the group labelled (O29) in \cite{EdjvetSwan_IrreducibleCyclicallyPresented}, nontriviality follows immediately since the exponent sum of \(x\) in the defining word $w(x,t)=(x^2t)^2x^2(tx)^{-4}$ is \(2\), rather than \(\pm 1\).
\end{rem}

We call $P(r,n,k,s,q)$ \emph{irreducible} if $\gcd(n,k-1,q)=1$. The type $\widetilde{\mathfrak{Z}}$ condition appears to capture an essential structural feature of irreducible perfect groups. Indeed, every known nontrivial perfect group $P(r,n,k,r-1,q)$ with $1\leq r<n$ that is irreducible satisfies the type $\widetilde{\mathfrak{Z}}$ condition. This observation motivates the following conjecture.

\begin{conj}[{\cite[Conjecture 1.3]{MR4315549}}]\label{conj}
Let \(n,k,q \geq 1\) and \(2 \leq r < n\) with \(\gcd(n,k-1,q)=1\).
Suppose \(k \not\equiv 1 \pmod{n}\) and \(k \not\equiv 1+q \pmod{n}\).
If \(P(r,n,k,r-1,q)\) is perfect, then it is of type \(\widetilde{\mathfrak{Z}}\).
\end{conj}

We are now ready to present the main classification theorem for the perfect groups \(P(r,n,k,s,q)\) in the case where \(\gcd(n,6)=1\).

\begin{theoremx}\label{main}
Let \(r,n,k,s,q\geq1\) with \(\gcd(n,6)=1\) and \(r\geq s\). If \(P(r,n,k,s,q)\) is perfect, then it is either trivially perfect or of type \(\widetilde{\mathfrak Z}\).
\end{theoremx}

\begin{rem}\label{rem:conjimplied}
Theorem~\ref{main} proves Conjecture~\ref{conj} when \(\gcd(n,6)=1\). Indeed, suppose \(P(r,n,k,r-1,q)\) is perfect and irreducible, with \(2\le r<n\), \(k\not\equiv1\) and \(k\not\equiv1+q \pmod n\). By \cite[Lemma~2.3]{MR4315549} perfectness gives \(d\mid k-1\), whence \(\gcd(n,k-1,q)=d\); irreducibility therefore forces \(d=1\). The last two alternatives of Definition~\ref{def:tperf} then read \(r\equiv0\) and \(r\equiv1 \pmod n\), both impossible for \(2\le r<n\). Hence the group is not trivially perfect, and Theorem~\ref{main} places it in type \(\widetilde{\mathfrak Z}\).
\end{rem}

\medskip
Suppose that \(w, w'\in F_n\) have the same exponent sum for each generator \(x_0,\dots,x_{n-1}\) (equivalently, they define the same element in the abelianisation of \(F_n\)). Then \(G_n(w)\) is perfect if and only if \(G_n(w')\) is perfect. The following result follows from Theorem~\ref{ThmC} and Theorem~\ref{main}.

\begin{corx}\label{maincor}
Let \(r,n,k,s,q \ge 1\) with \(\gcd(n,6)=1\) and \(r\ge s\). Let \(G_n(w)\) be a cyclically presented group, and suppose there is a Prishchepov group \(P(r,n,k,s,q)=G_n(w')\) such that \(w=w'\) in \(F_n^{\mathrm{ab}}\). Then \(G_n(w)\) is perfect if and only if \(s=r-1\) and one of the following holds:
\begin{enumerate}
    \item \(P(r,n,k,s,q)\) is trivially perfect;
    \item \(P(r,n,k,s,q)\) is of type \(\widetilde{\mathfrak{Z}}\) and
    \(\gcd(k-1-qr,n)=\gcd(q,n)\).
\end{enumerate}
\end{corx}

The condition \(\gcd(n,6)=1\) excludes the cases in which \(n\) is divisible by \(2\) or \(3\). However, for fixed \(r\) and \(k\), trivial perfectness lifts from \(P(r,n/2,k,r-1,1)\) to \(P(r,n,k,r-1,1)\) whenever \(n\) is even and \(P(r,n,k,r-1,1)\) is perfect. Consequently, by repeatedly factoring out powers of \(2\), the study of perfectness for groups with \(2\mid n\) reduces to the corresponding odd component of \(n\).

\begin{theoremx}\label{main2}
Let \(r,n,m,M,k,s,q\geq1\) with \(n=2^mM\), where \(\gcd(M,6)=1\), and suppose that \(r\geq s\). If \(P(r,n,k,s,q)\) is perfect, then one of the following holds:
\begin{enumerate}
    \item \(P(r,n,k,s,q)\) is trivially perfect;
    \item \(P(r,M,k,s,q)\) is of type \(\widetilde{\mathfrak Z}\).
\end{enumerate}
\end{theoremx}

The next corollary provides a complete characterisation of perfectness whenever \(n\) has no odd prime divisors other than \(3\), with \(3\) occurring to the first power at most. (It is not a special case of Theorem~\ref{main2}, which excludes \(3\mid n\); it is deduced instead from Corollary~\ref{lcor}.) This is consistent with Conjecture~\ref{conj}, which implies that the same conclusion holds for all even values of \(n\), assuming \(P(r,n,k,s,q)\) is irreducible.

\begin{corx}\label{main2cor}
Let \(n=2^a3^b\), where \(a\geq0\) and \(b\in\{0,1\}\), and suppose that \(r\geq s\). Then \(P(r,n,k,s,q)\) is perfect if and only if it is trivially perfect.
\end{corx}
The restriction \(b\leq1\) in Corollary \ref{main2cor} is sharp and is not merely an artefact of the proof. Indeed, when \(9\mid n\), the conclusion of the corollary fails in general. For example,
$
P(5,9,4,4,1)
$ is perfect, with \(R_9(F)=1\), although it is not trivially perfect (\(r\equiv5\) and \(k\equiv4\pmod 9\)). It is of type \(\mathfrak Z'\), since \(2r=10\equiv1\pmod9\), and
$$
\gcd(k-1-r,9)=\gcd(-2,9)=1.
$$
Similarly,
$
P(3,9,6,2,1)$ and $P(4,9,6,3,1)
$
are perfect examples of type \(\mathfrak Z\), since \(2(k-1)=10\equiv1\pmod9\). Hence, the hypothesis \(b\leq1\) cannot, in general, be removed.

\medskip
In the Prishchepov setting, additional structure is often needed to identify trivial groups among the perfect ones.

\medskip
\subsection*{Structure}
The remainder of this article is organised as follows. In Section~\ref{sec:Method}, we describe the method used in determining the abelianisation of $G_n(w)$. Section~\ref{sec:Odd} is devoted to the technical results required for the proof of Theorem~\ref{main}, while Section~\ref{sec:Even} treats the corresponding technical results for Theorem~\ref{main2}. Finally, in Section~\ref{sec:Main}, we include a reduction to a normalised case and a lifting argument to recover the general case, as well as give proofs of the main results.

\section{Abelianisation}\label{sec:Method}

This section develops the algebraic framework used to determine when a cyclically presented group is perfect. We first explain how to compute the abelianisation of a cyclically presented group. The computation naturally leads to circulant matrices and their associated polynomials. We then relate these polynomials to the defining presentations of the groups \(P(r,n,k,s,q)\).

\medskip
To determine whether \(G_n(w)\) is perfect, we study its abelianisation.
Let \(c_i\) denote the exponent sum of \(x_i\) in the defining word \(w\),
for \(0\leq i<n\). The relation matrix of \(G_n(w)^{\mathrm{ab}}\) is the
circulant matrix
\[
C=\operatorname{circ}_n(c_0,c_1,\ldots,c_{n-1}).
\]
The group \(G_n(w)\) is perfect if and only if \(C\) is unimodular, while
\(G_n(w)^{\mathrm{ab}}\) is infinite if and only if \(C\) is singular. Associated with \(C\) is the circulant polynomial
\[
f_C(t)=\sum_{i=0}^{n-1}c_it^i.
\]
By \cite[Equation 3.2.14]{MR543191} and
\cite[Theorem 3, p.~78]{johnson1980topics},
\[
\det(C)=R_n(f_C)=\prod_{i=0}^{n-1}f_C(\zeta^i),
\]
where \(\zeta\) is a primitive \(n\)th root of unity. Hence \(C\) is
unimodular precisely when \(R_n(f_C)=\pm1\), and singular precisely when \(R_n(f_C)=0\).

\medskip
For the Prishchepov family \(P(r,n,k,s,q)\), the associated polynomial is
\[
f(t)=
\sum_{i=0}^{r-1}t^{qi}
-
t^{k-1}\sum_{i=0}^{s-1}t^{qi} \pmod{t^n-1}.
\]
By the symmetry and reduction results in Section~\ref{normal}, the study of perfectness may
be reduced to the normalised case \(s=r-1\) and \(q=1\). The corresponding associated polynomial for this reduced case is
\begin{equation}\label{eq2}
F(t)=
\sum_{i=0}^{r-1}t^i
-
t^{k-1}\sum_{i=0}^{r-2}t^i .
\end{equation}
We shall also use the polynomial
\[
G(t)=
\sum_{i=0}^{k-2}t^i
-
t^r\sum_{i=0}^{k-3}t^i ,
\]
obtained from \(F(t)\) by the involution
\((k,r)\mapsto(r+1,k-1)\). By
\cite[Lemma~2.6]{MR4315549}, the corresponding resultants agree.

\section{Technical results for the case $\gcd(n,6)=1$}\label{sec:Odd}

In this section we collect the technical tools needed to prove Theorem~\ref{main}. Fix an integer $n \geq 1$, and let $F(t)$ denote the polynomial defined in~\eqref{eq2} with integer parameters $r \geq 2$ and $k \geq 1$.
Let $\zeta$ be a primitive $n$th root of unity.
Our goal is to analyse solutions to the equation
\begin{equation}\label{eq4}
 \epsilon\zeta^j F(\zeta) \;=\; F(\zeta^{-1}),
\end{equation}
for some $0 \leq j < n$ and $\epsilon=\pm1$.

\medskip
We begin by recording, for later use, that the parameters $r$ and $k$ may always be reduced modulo $n$. This was used implicitly in the original treatment; since all of the trigonometric estimates in Section~\ref{sec:Even} require $r<n$, we state it explicitly.

\begin{lem}\label{lem:reduce}
Let $n\ge1$ and let $r,k\ge1$, $r'\ge 1$, $k'\ge1$ satisfy $r'\equiv r$ and $k'\equiv k \pmod n$. Write $F_{r,k}(t)$ for the polynomial \eqref{eq2}. Then $F_{r',k'}(\eta)=F_{r,k}(\eta)$ for every $n$th root of unity $\eta$. In particular $R_n(F_{r',k'})=R_n(F_{r,k})$, so $P(r,n,k,r-1,1)$ is perfect if and only if $P(r',n,k',r'-1,1)$ is perfect.
\end{lem}

\begin{proof}
For $\eta\ne1$ with $\eta^{n}=1$ we have $\sum_{i=0}^{n-1}\eta^{i}=0$, so
$\sum_{i=0}^{m-1}\eta^{i}=(1-\eta^{m})/(1-\eta)$ depends only on $m$ modulo $n$; moreover $\eta^{k-1}$ depends only on $k$ modulo $n$. Hence $F_{r,k}(\eta)$ depends only on the residues of $r$ and $k$. For $\eta=1$ we have $F_{r,k}(1)=r-(r-1)=1=F_{r',k'}(1)$.
\end{proof}

Accordingly, throughout Sections~\ref{sec:Odd} and~\ref{sec:Even} we assume,
whenever convenient, that $1\le r\le n$ and $1\le k\le n$ (representatives in
$[1,n]$, so that both sums in \eqref{eq2} are defined). When $r\not\equiv 0,1$
and $k\not\equiv 1,2 \pmod n$ this gives $2\le r\le n-1$ and $3\le k\le n-1$.

\medskip
We next recall a classical fact that allows us to identify multisets of complex numbers from their power sums.

\begin{lem}[Newton--Girard]\label{N-G}
Let $w_1,\dots,w_\ell, z_1,\dots,z_\ell \in \mathbb{C}$. If
\[
\sum_{i=1}^{\ell} w_i^j = \sum_{i=1}^{\ell} z_i^j \quad \text{for all } j=1,\dots,\ell,
\]
then the multisets $\{w_1,\dots,w_\ell\}$ and $\{z_1,\dots,z_\ell\}$ are equal.
\end{lem}

Intuitively, the power sums determine the elementary symmetric polynomials via the Newton--Girard identities (the recursion is invertible in characteristic $0$), so equality of the first $\ell$ power sums implies equality of the corresponding multisets of roots.

\medskip
To our knowledge, the first explicit use of Lemma~\ref{N-G} in the context of cyclically presented groups is due to Odoni \cite{Odoni}, who applied it to a trinomial. A similar application to a trinomial was later given in \cite{Will1}. The present work extends these earlier applications by applying the lemma to a polynomial of arbitrary degree, rather than to a trinomial.

\medskip
Next, we state a criterion for the unimodularity of circulant matrices in terms of cyclotomic units.

\begin{lem}[\cite{MR2398785}, Lemma 3 and Corollary]\label{Primitive}
Let $C$ be the $n\times n$ circulant matrix with first row $(a_0,\dots,a_{n-1})$ and associated polynomial
\(
p(t) = a_0 + a_1 t + \cdots + a_{n-1} t^{\,n-1}.
\)
Then $C$ is unimodular if and only if
\(
p(\zeta_d) \in \mathbb{Z}[\zeta_d]^\times\) for every divisor $d$ of $n$, where $\zeta_d$ is a primitive $d$th root of unity.
\end{lem}

The key link to our setting comes from a result of Odoni, which shows that if a polynomial evaluates to a cyclotomic unit, then its values at $\zeta$ and $\zeta^{-1}$ are related exactly as in~\eqref{eq4}.

\begin{lem}[\cite{Odoni}, proof of Lemma~3.1]\label{OdoniCriterion}
Let $n > 1$ be an integer, $\zeta$ a primitive $n$th root of unity, and $g(t)$ a polynomial.
If $g(\zeta)$ is a unit in $\mathbb{Z}[\zeta]$, then there exist an integer $j$ and $\epsilon \in \{\pm 1\}$ such that
\(
\epsilon\zeta^j g(\zeta) = g(\zeta^{-1}).
\)
\end{lem}

\begin{rem}\label{epsil}
When $n$ is even, $-1$ is an $n$th root of unity, so $\epsilon\zeta^j$ ranges over the same set of $n$th roots of unity as $\zeta^j$. Therefore, the factor $\epsilon$ may be omitted in Lemma~\ref{OdoniCriterion}.
\end{rem}

We now combine these ingredients. We shall apply Lemma~\ref{N-G} with $\ell = 4$, all $w_i, z_i$ being $n$th roots of unity. The following observation is what makes this legitimate.

\begin{lem}\label{lem:autos}
Suppose $\gcd(n,6)=1$ and let $\zeta$ be a primitive $n$th root of unity. For $m\in\{1,2,3,4\}$ the assignment $\sigma_m(\zeta)=\zeta^{m}$ extends to a field automorphism of $\mathbb Q(\zeta)$. Consequently, if $Z$ and $W$ are multisets of $4$ residues modulo $n$ with
\[
\sum_{z\in Z}\zeta^{z}=\sum_{w\in W}\zeta^{w},
\]
then $Z=W$ as multisets modulo $n$.
\end{lem}

\begin{proof}
Since $\gcd(n,6)=1$, the integer $n$ is odd and not divisible by $3$; hence $\gcd(n,m)=1$ for $m=1,2,3,4$, and $\sigma_m\in\operatorname{Gal}(\mathbb Q(\zeta)/\mathbb Q)$. Applying $\sigma_m$ to the displayed identity gives
\[
\sum_{z\in Z}(\zeta^{z})^{m}=\sum_{w\in W}(\zeta^{w})^{m},\quad m=1,2,3,4,
\]
that is, the first four power sums of the two $4$-element multisets $\{\zeta^{z}\}_{z\in Z}$ and $\{\zeta^{w}\}_{w\in W}$ agree. By Lemma~\ref{N-G}, these multisets of complex numbers coincide. Since $\zeta$ has order $n$, the map $z\mapsto\zeta^{z}$ is a bijection from
$\mathbb Z/n$ onto the $n$th roots of unity; hence equality of the multisets
$\{\zeta^{z}\}_{z\in Z}$ and $\{\zeta^{w}\}_{w\in W}$ implies $Z=W$ as
multisets modulo $n$.
\end{proof}

\begin{rem}
Lemma~\ref{lem:autos} is exactly where the hypothesis $\gcd(n,6)=1$ enters, and it is the reason a separate treatment is required when $2\mid n$ or $3\mid n$: for such $n$ the map $\zeta\mapsto\zeta^{2}$ (resp. $\zeta\mapsto\zeta^{3}$) is not injective on $n$th roots of unity, and equality of the first four power sums no longer forces equality of the multisets. Section~\ref{sec:Even} develops an alternative argument for even $n$.
\end{rem}

Under these conditions, equation~\eqref{eq4} imposes strong arithmetic restrictions, leading to the following key lemma.

\begin{lem}\label{mainlem}
Let $n,k\geq 1$ and $r\geq 2$ be integers and let $F(t)$ be as in \eqref{eq2}. Suppose the identity
\(\epsilon\zeta^j F(\zeta) = F(\zeta^{-1})\)
holds for some integer $0 \le j < n$ and $\epsilon \in \{\pm 1\}$, where $\zeta$ is a primitive $n$th root of unity.
If $\gcd(n,6)=1$, $r \not\equiv 0,1 \bmod{n}$, and $k \not\equiv 1,2 \bmod{n}$,
then either
\(
2r \equiv 1 \bmod{n}\) or \(2k \equiv 3 \bmod{n}.
\)
\end{lem}

\begin{proof}
The assumptions $r \not\equiv 0,1$ and $k \not\equiv 1,2 \pmod n$ force $n>1$. Recall that
\[
F(t) = \sum_{i=0}^{r-1} t^i - t^{k-1} \sum_{i=0}^{r-2} t^i
= \frac{1 - t^{k-1} - t^r + t^{k+r-2}}{1-t}, \quad t \neq 1.
\]
Since $n > 1$, we may substitute $t = \zeta$ in \eqref{eq4}. Multiplying through by $(1-\zeta)$ and using $1-\zeta^{-1}=-\zeta^{-1}(1-\zeta)$, we obtain
\begin{equation}\label{Id}
\epsilon \zeta^j \bigl(1 - \zeta^{k-1} - \zeta^r + \zeta^{k+r-2}\bigr)
= -\zeta \bigl(1 - \zeta^{-(k-1)} - \zeta^{-r} + \zeta^{-(k+r-2)}\bigr).
\end{equation}

For $\epsilon=+1$, equation \eqref{Id} rearranges to
\begin{equation}\label{Identity}
\zeta^j + \zeta^{j+k+r-2} + \zeta + \zeta^{-(k+r-3)}
= \zeta^{j+k-1} + \zeta^{j+r} + \zeta^{-(k-2)} + \zeta^{1-r}.
\end{equation}
For $\epsilon=-1$, equation \eqref{Id} becomes
\begin{equation}\label{Identity2}
\zeta^j + \zeta^{j+k+r-2} + \zeta^{2-k} + \zeta^{1-r}
= \zeta^{j+k-1} + \zeta^{j+r} + \zeta + \zeta^{3-k-r}.
\end{equation}

In each case, both sides are sums of four $n$th roots of unity, so by Lemma~\ref{lem:autos} the identity is equivalent to the equality
\[
Z \equiv W \pmod{n}
\]
of the multisets $Z$ and $W$ of exponents appearing on the left- and right-hand sides.

\medskip
We now make these multisets explicit and analyse the two cases $\epsilon=\pm1$ separately. In both cases we use repeatedly that the element $j\in Z$ can match neither $j+k-1$ (which would force $k\equiv1$) nor $j+r$ (which would force $r\equiv0$), both of which are excluded by hypothesis.

\subsection*{Case $\epsilon = +1$}
The exponent multisets, modulo $n$, are
\[
Z = \{ j,\, j+k+r-2,\, 1,\, -(k+r-3) \}, \quad
W = \{ j+k-1,\, j+r,\, -(k-2),\, 1-r \}.
\]
Hence, the only possibilities are
\(
j \equiv 1-r\) or \(j \equiv 2-k \pmod n,
\)
and these are interchanged by the involution $(k,r)\mapsto(r+1,k-1)$. Taking $j \equiv 1-r$ gives
\[
Z = \{1-r,\; k-1,\; 1,\; 3-k-r\}, \quad W = \{k-r,\; 1,\; 2-k,\; 1-r\} \pmod{n},
\]
and cancelling $1$ and $1-r$ leaves $\{k-1,\,3-k-r\}$ and $\{k-r,\,2-k\}$. If $k-1\equiv k-r$ then $r\equiv1$, which is excluded; hence
\(
k-1 \equiv 2-k\), that is, \(2k \equiv 3 \pmod n.
\)
For the alternative $j\equiv 2-k$ we use the symmetry of the numerator:
$1-t^{k-1}-t^{r}+t^{k+r-2}$ is unchanged under $(k,r)\mapsto(r+1,k-1)$, so
$F(\zeta)$ and $F(\zeta^{-1})$, and hence \eqref{eq4}, are unchanged as well.
The new parameters $r'=k-1$, $k'=r+1$ satisfy the hypotheses
($r'\not\equiv0,1$ since $k\not\equiv1,2$, and $k'\not\equiv1,2$ since
$r\not\equiv0,1$), and $j\equiv 2-k=1-r'$. The case just treated then gives
$2k'\equiv3$, that is, $2r\equiv1\pmod n$.

\subsection*{Case $\epsilon = -1$}
The exponent multisets, modulo $n$, are
\[
Z = \{ j,\, j+k+r-2,\, 2-k,\, 1-r \}, \quad
W = \{ j+k-1,\, j+r,\, 1,\, 3-k-r \}.
\]
Hence the candidates are $j \equiv 1$ or $j \equiv 3-k-r \pmod n$.

\medskip
Suppose first that $j \equiv 1$. Then
\[
Z = \{1,\; k+r-1,\; 2-k,\; 1-r\}, \quad W = \{k,\; r+1,\; 1,\; 3-k-r\} \pmod{n},
\]
and cancelling $1$ leaves $\{k+r-1,\,2-k,\,1-r\}$ and $\{k,\,r+1,\,3-k-r\}$. Matching $k+r-1$ with $k$ forces $r\equiv1$, and with $r+1$ forces $k\equiv2$; both are excluded. Hence $k+r-1\equiv3-k-r$, that is, $2(k+r)\equiv4 \pmod n$. As $n$ is odd this gives $k+r\equiv2$, and substituting $r\equiv2-k$ the two remaining multisets become
\[
\{2-k,\;k-1\}\quad\text{and}\quad\{k,\;3-k\} \pmod n .
\]
If $2-k\equiv k$ then $2k\equiv2$, so $k\equiv1$ ($n$ odd), which is excluded; and if $2-k\equiv3-k$ then $n\mid1$, which is impossible. So no solution arises here.

\medskip
Suppose next that $j \equiv 3-k-r$. Then
\[
Z = \{3-k-r,\; 1,\; 2-k,\; 1-r\}, \quad W = \{2-r,\; 3-k,\; 1,\; 3-k-r\} \pmod{n},
\]
and cancelling $1$ and $3-k-r$ leaves $\{2-k,\,1-r\}$ and $\{2-r,\,3-k\}$. If $2-k\equiv3-k$ then $n\mid 1$. Otherwise $2-k\equiv2-r$ and $1-r\equiv3-k$, giving $k\equiv r$ and then $1-r\equiv3-r$, that is, $n\mid2$; as $n$ is odd and $n>1$, this is impossible.

\medskip
Hence no solution arises in the case $\epsilon=-1$, and the lemma is proved.
\end{proof}

\section{Technical results for the even case}\label{sec:Even}

Assume throughout this section that we are working with \(P(r,n,k,r-1,1)\), with \(n\) even. The purpose of the section is to show that any congruence \(k \equiv 1\) or \(2 \pmod{n/2}\) lifts to \(k \equiv 1\) or \(2 \pmod{n}\) when \(P(r,n,k,r-1,1)\) is perfect; equivalently, to exclude the shifted values \(k\equiv (n+2)/2\) and \(k\equiv(n+4)/2 \pmod n\).

\begin{lem}\label{upart}
Let \(n\ge 2\) be even, let \(r\ge 2\), and let \(\zeta\) be a primitive \(n\)th root of unity. Let \(F(t)\) be as in \eqref{eq2}, where
\[
k\equiv u+\tfrac{an}{2}\pmod{n},
\quad 1\le u\le \tfrac{n}{2},\quad a \text{ odd}.
\]
Suppose that \(F(\zeta)=\zeta^{j}F(\zeta^{-1})\) for some integer \(j\). Then
\[
(1-\zeta)\bigl(\zeta^{r-1}-\zeta^{\,j-(r-1)}\bigr)
=
(1-\zeta^{r-1})
\bigl(\zeta^{\,j-(r-2)}-\zeta^{u-1}\bigr)
\bigl(1+\zeta^{-(u-1)}\bigr).
\]
\end{lem}

\begin{proof}
Write
\[
F(t)=t^{r-1}+S(t)\bigl(1-t^{k-1}\bigr), \quad
S(t)=\frac{1-t^{r-1}}{1-t}=1+t+\cdots+t^{r-2}.
\]
Since $k\equiv u+\frac{an}{2}$ with $a$ odd and $\zeta^{n/2}=-1$, we have
$\zeta^{k-1}=-\zeta^{u-1}$ and $\zeta^{-(k-1)}=-\zeta^{-(u-1)}$. Moreover
\[
S(\zeta^{-1})=\frac{1-\zeta^{-(r-1)}}{1-\zeta^{-1}}
=\zeta^{-(r-2)}\,\frac{1-\zeta^{\,r-1}}{1-\zeta}=\zeta^{-(r-2)}S(\zeta).
\]
Hence
\[
F(\zeta)=\zeta^{r-1}+S(\zeta)\bigl(1+\zeta^{u-1}\bigr),
\quad
F(\zeta^{-1})=\zeta^{-(r-1)}+\zeta^{-(r-2)}S(\zeta)\bigl(1+\zeta^{-(u-1)}\bigr).
\]
Imposing $F(\zeta)=\zeta^j F(\zeta^{-1})$ and rearranging gives
\[
\zeta^{r-1}-\zeta^{\,j-(r-1)}
=S(\zeta)\Bigl[\zeta^{\,j-(r-2)}\bigl(1+\zeta^{-(u-1)}\bigr)-\bigl(1+\zeta^{u-1}\bigr)\Bigr].
\]
Since $1+\zeta^{u-1}=\zeta^{u-1}\bigl(1+\zeta^{-(u-1)}\bigr)$, the bracket equals
$\bigl(1+\zeta^{-(u-1)}\bigr)\bigl(\zeta^{\,j-(r-2)}-\zeta^{u-1}\bigr)$.
Multiplying by $1-\zeta$ and using $S(\zeta)(1-\zeta)=1-\zeta^{r-1}$ yields the stated identity.
\end{proof}

We isolate the one elementary trigonometric inequality that is needed.

\begin{lem}\label{lem:sin-ineq}
Let $n\ge 8$ and $3\le r<n$ be integers, and set $\alpha=\pi/n$. Then
\[
2\sin\bigl((r-1)\alpha\bigr)>\sin\bigl((r+1)\alpha\bigr).
\]
\end{lem}

\begin{proof}
Since $3\le r<n$ we have $0<(r-1)\alpha<(r+1)\alpha\le\pi$, so both sines are non-negative and $\sin((r-1)\alpha)>0$.

\medskip
\emph{Case 1: $(r+1)\alpha\le\pi/2$.} On $(0,\pi/2]$ the function $x\mapsto\sin(x)/x$ is strictly decreasing, so
\[
\frac{\sin((r-1)\alpha)}{r-1}>\frac{\sin((r+1)\alpha)}{r+1},
\quad\text{i.e.}\quad
\sin\bigl((r+1)\alpha\bigr)<\frac{r+1}{r-1}\,\sin\bigl((r-1)\alpha\bigr).
\]
As $r\ge3$ we have $(r+1)/(r-1)\le2$, and the claim follows.

\medskip
\emph{Case 2: $(r-1)\alpha\ge\pi/2$.} Then $\pi/2\le(r-1)\alpha<(r+1)\alpha\le\pi$, where $\sin$ is decreasing, so $\sin((r+1)\alpha)\le\sin((r-1)\alpha)<2\sin((r-1)\alpha)$.

\medskip
\emph{Case 3: $(r-1)\alpha<\pi/2<(r+1)\alpha$.} Then $(r-1)\alpha>\pi/2-2\alpha$, and since $n\ge8$ we have $2\alpha\le\pi/4$, so $\pi/2-2\alpha\in[\pi/4,\pi/2)$. As $\sin$ is increasing on $[0,\pi/2]$,
\[
\sin\bigl((r-1)\alpha\bigr)>\sin\Bigl(\frac{\pi}{2}-2\alpha\Bigr)=\cos(2\alpha)\ge\cos\Bigl(\frac{\pi}{4}\Bigr)=\frac{\sqrt2}{2}.
\]
Hence $2\sin((r-1)\alpha)>\sqrt2>1\ge\sin((r+1)\alpha)$.
\end{proof}

We now prove the main result of this section, which is the analytic input for Theorem~\ref{double}.

\begin{lem}\label{noj}
Let $n\ge 8$ and $3\le r<n$ be integers, let $u\in\{1,2\}$, and let $0\le j<n$. Set $\alpha=\pi/n$ and
\[
A=\sin\bigl((r-1)\alpha\bigr),\quad
B_u=\bigl|\sin\bigl((j-(r+u-3))\alpha\bigr)\bigr|,\quad
C=\sin(\alpha),
\]
\[
D=\bigl|\sin\bigl((j-2r+2)\alpha\bigr)\bigr|,\quad
E_u=\cos\bigl((u-1)\alpha\bigr).
\]
Then $2AB_uE_u\ne CD$.
\end{lem}

\begin{proof}
Put
\[
b=r+u-3,\quad c=r-u+1,\quad \Omega_u=\frac{2AE_u}{C}=\frac{2\sin((r-1)\alpha)\cos((u-1)\alpha)}{\sin\alpha},
\]
and let $m\in\{0,1,\dots,n-1\}$ be the residue of $j-b$ modulo $n$. Note that
\[
3\le r<n \quad\Longrightarrow\quad 1\le b\le n-2,\quad 2\le c\le n-1,
\]
and $A>0$, $C>0$, $E_u>0$ (the last because $0\le(u-1)\alpha\le\alpha<\pi/2$). Also $2r-2-b\equiv c\pmod n$, so
\[
D=\bigl|\sin\bigl((m-c)\alpha\bigr)\bigr|,\quad B_u=\bigl|\sin(m\alpha)\bigr| .
\]

\medskip
\emph{The case $m=0$.} Then $B_u=0$, so $2AB_uE_u=0$, while $D=|\sin(c\alpha)|>0$ because $1\le c\le n-1$; hence $CD>0$ and the two sides differ.

\medskip
\emph{The case $m\ne0$.}
\medskip
Then $\sin(m\alpha)>0$, and $2AB_uE_u=CD$ holds if and only if
\[
h(m):=\frac{D}{B_u}=\frac{\bigl|\sin((m-c)\alpha)\bigr|}{\sin(m\alpha)}=\Omega_u .
\]
We show that $h(m)<\Omega_u$ for every integer $m$ with $1\le m\le n-1$, which proves the lemma. We regard $h$ as a function of a real variable on the two subintervals below; on each of them the sign of $\sin((m-c)\alpha)$ is constant, so no sign changes are hidden.

\medskip
\emph{Subinterval $1\le m\le c$.} 

\medskip
Here $0\le(c-m)\alpha\le(c-1)\alpha\le(n-2)\alpha<\pi$, so
$h(m)=\sin((c-m)\alpha)/\sin(m\alpha)$. Differentiating,
\[
h'(m)=\alpha\cdot\frac{-\cos((c-m)\alpha)\sin(m\alpha)-\sin((c-m)\alpha)\cos(m\alpha)}{\sin^{2}(m\alpha)}
=\frac{-\alpha\sin(c\alpha)}{\sin^{2}(m\alpha)}<0,
\]
since $\sin(c\alpha)>0$. Hence, $h$ is strictly decreasing and
\[
h(m)\le h(1)=\frac{\sin((c-1)\alpha)}{\sin\alpha}.
\]

\medskip
\emph{Subinterval $c\le m\le n-1$.} 

\medskip
Here $0\le(m-c)\alpha\le(n-1-c)\alpha<\pi$, so
$h(m)=\sin((m-c)\alpha)/\sin(m\alpha)$, and the same computation gives
$h'(m)=\alpha\sin(c\alpha)/\sin^{2}(m\alpha)>0$. Hence, $h$ is strictly increasing and, using
$(n-1-c)\alpha=\pi-(c+1)\alpha$ and $(n-1)\alpha=\pi-\alpha$,
\[
h(m)\le h(n-1)=\frac{\sin((n-1-c)\alpha)}{\sin((n-1)\alpha)}=\frac{\sin((c+1)\alpha)}{\sin\alpha}.
\]

\medskip
Therefore, for all $1\le m\le n-1$,
\begin{equation}\label{eq:hbound}
h(m)\le\frac{\max\bigl\{\sin((c-1)\alpha),\;\sin((c+1)\alpha)\bigr\}}{\sin\alpha}.
\end{equation}

\medskip
\emph{Conclusion when $u=2$.} Here $c=r-1$, so $c-1=r-2\in[1,n-3]$ and $c+1=r\in[3,n-1]$; hence $\sin((c\pm1)\alpha)>0$. By the product-to-sum identity,
\[
\Omega_2=\frac{2\sin((r-1)\alpha)\cos\alpha}{\sin\alpha}
=\frac{\sin((c+1)\alpha)+\sin((c-1)\alpha)}{\sin\alpha},
\]
which strictly exceeds each of the two quantities in \eqref{eq:hbound}. Hence $h(m)<\Omega_2$.

\medskip
\emph{Conclusion when $u=1$.} Here $c=r$, so
\(
\Omega_1=2\sin((c-1)\alpha)/\sin\alpha,
\)
with $c-1=r-1\in[2,n-2]$, so $\sin((c-1)\alpha)>0$ and therefore
$\sin((c-1)\alpha)<\Omega_1\sin\alpha$. Furthermore, by Lemma~\ref{lem:sin-ineq},
\[
\sin\bigl((c+1)\alpha\bigr)=\sin\bigl((r+1)\alpha\bigr)<2\sin\bigl((r-1)\alpha\bigr)=\Omega_1\sin\alpha .
\]
Both quantities in \eqref{eq:hbound} are therefore strictly smaller than $\Omega_1\sin\alpha$, so $h(m)<\Omega_1$.
\end{proof}

We next dispose of the small parameters, which are needed as the base of the doubling argument and which were previously asserted without proof.

\begin{lem}\label{lem:r2}
Let $M\ge2$ and $n=2M$. Suppose $r\equiv2\pmod n$ and
$k\equiv u+aM\pmod n$ with $u\in\{1,2\}$ and $a$ odd. Then
$P(r,n,k,r-1,1)$ is not perfect.
\end{lem}

\begin{proof}
By Lemma~\ref{lem:reduce} we may take $r=2$, so $F(t)=1+t-t^{k-1}$. Since $a$ is
odd, $\zeta^{aM}=\zeta^{M}\in\{\pm1\}$ for every $n$th root of unity $\zeta$.

If $u=1$, then $\zeta^{k-1}=\zeta^{M}$, so $F(\zeta)=\zeta$ when $\zeta^{M}=1$
and $F(\zeta)=2+\zeta$ when $\zeta^{M}=-1$. Hence
\[
R_n(F)=\Bigl(\prod_{\zeta^{M}=1}\zeta\Bigr)\Bigl(\prod_{\zeta^{M}=-1}(2+\zeta)\Bigr).
\]
The first factor is a root of unity. Since $\prod_{\zeta^{M}=-1}(x-\zeta)=x^{M}+1$,
putting $x=-2$ gives
\[
\prod_{\zeta^{M}=-1}(2+\zeta)=(-1)^{M}\bigl((-2)^{M}+1\bigr)=2^{M}+(-1)^{M}.
\]
So $|R_n(F)|=2^{M}+(-1)^{M}\ge3$ for $M\ge2$.

If $u=2$, then $\zeta^{k-1}=\zeta\cdot\zeta^{M}$, so $F(\zeta)=1$ when
$\zeta^{M}=1$ and $F(\zeta)=1+2\zeta$ when $\zeta^{M}=-1$. Then
\[
|R_n(F)|=\Bigl|\prod_{\zeta^{M}=-1}(1+2\zeta)\Bigr|
=2^{M}\Bigl|\prod_{\zeta^{M}=-1}\bigl(\tfrac12+\zeta\bigr)\Bigr|
=\bigl|1+(-1)^{M}2^{M}\bigr|=2^{M}+(-1)^{M}\ge3
\]
for $M\ge2$. (Here $\prod_{\zeta^{M}=-1}(\tfrac12+\zeta)=(-1)^{M}\bigl((-\tfrac12)^{M}+1\bigr)$.)
In both cases $|R_n(F)|>1$, so the group is not perfect.
\end{proof}

The hypothesis $M\ge2$ in Lemma \ref{lem:r2} is necessary: for $M=1$ one gets $|R_2(F)|=1$.

\begin{lem}\label{lem:small}
Let $n\in\{2,4,6\}$. Then $P(r,n,k,r-1,1)$ is perfect if and only if it is trivially perfect, that is, if and only if $r\equiv0,1$ or $k\equiv1,2 \pmod n$.
\end{lem}

\begin{proof}
By Lemma~\ref{lem:reduce}, it suffices to check the finitely many residue classes of $r,k$ modulo $n$; in the table below residues are written in $\{0,\dots,n-1\}$, with $0$ standing for the representative $n$. If $r\equiv0,1$ or $k\equiv1,2$ the group is trivially perfect, hence perfect. For the remaining pairs we compute $|R_n(F)|$ directly; the involution $(k,r)\mapsto(r+1,k-1)$ of \cite[Lemma~2.6]{MR4315549} preserves $R_n(F)$ and is used to halve the work.

\medskip
For $n=2$ every pair is trivially perfect, since $r\equiv0$ or $r\equiv1 \pmod 2$. For $n=4$ and $n=6$, the remaining values are listed below; in every case $|R_n(F)|>1$, so the group is not perfect.

\begin{center}
\begin{tabular}{llc}
\toprule
$n$ & $(r,k)$ & $|R_n(F)|$\\
\midrule
$4$ & $(2,3),\,(2,0),\,(3,3),\,(3,0)$ & $5$\\
\midrule
$6$ & $(2,3),\,(2,0),\,(5,3),\,(5,0)$ & $16$\\
$6$ & $(2,4),\,(2,5),\,(3,3),\,(3,0),\,(4,3),\,(4,0),\,(5,4),\,(5,5)$ & $7$\\
$6$ & $(3,4),\,(3,5),\,(4,4),\,(4,5)$ & $13$\\
\bottomrule
\end{tabular}
\end{center}

For instance, for $n=6$, $(r,k)=(2,3)$ gives $F(t)=1+t-t^{2}$, and since
$\prod_{\zeta^{n}=1}\bigl(1+\zeta-\zeta^{2}\bigr)=(-1)^{n}\bigl((-1)^{n}-L_{n}+1\bigr)$
with $L_n$ the $n$th Lucas number, we get $2-L_6=2-18=-16$. Similarly $(r,k)=(3,4)$ gives $F(t)=1+t+t^{2}-t^{3}-t^{4}$, with $F(1)=1$, $F(-1)=1$, $F(\rho)=\rho^{2}$ for $\rho$ a primitive cube root of unity, and $|F(\omega)|^{2}=13$ for $\omega$ a primitive sixth root of unity, so $|R_6(F)|=13$.
\end{proof}

\begin{theorem}\label{double}
Let \(n\ge2\) be even and let \(k,r\ge 1\). If \(P(r,n,k,r-1,1)\) is perfect, then it is trivially perfect if and only if \(P(r,n/2,k,r-1,1)\) is trivially perfect.
\end{theorem}

\begin{proof}
Write $n=2M$. The implication from $n$ to $M$ is immediate on reducing the congruences of Definition~\ref{def:tperf} modulo $M$. We therefore prove the doubling implication: if \(P(r,M,k,r-1,1)\) is trivially perfect and \(P(r,2M,k,r-1,1)\) is perfect, then \(P(r,2M,k,r-1,1)\) is trivially perfect.

\medskip
By Lemma~\ref{lem:small} we may assume $2M\ge8$. By Lemma~\ref{lem:reduce} we may assume $1\le r\le 2M$ and $1\le k\le 2M$. 

\medskip
Assume that $P(r,M,k,r-1,1)$ is trivially perfect, i.e.\ $r\equiv0,1$ or
$k\equiv1,2\pmod M$. The involution $(k,r)\mapsto(r+1,k-1)$ of
\cite[Lemma~2.6]{MR4315549} preserves perfectness, and it interchanges the
conditions ``$r\equiv0,1$'' and ``$k\equiv1,2$'' (modulo $M$ and modulo $2M$
alike), hence preserves trivial perfectness at both levels. If
$k\not\equiv1,2\pmod M$ then $r\equiv0,1\pmod M$, and after applying the
involution (and re-reducing as in Lemma~\ref{lem:reduce}) we have
$k\equiv1,2\pmod M$. So we may assume $k\equiv1,2\pmod M$.

\medskip
Suppose, for contradiction, that $P(r,2M,k,r-1,1)$ is not trivially perfect.
Then $r\not\equiv0,1$ and $k\not\equiv1,2\pmod{2M}$. Since $k\equiv u\pmod M$
for some $u\in\{1,2\}$ but $k\not\equiv u\pmod{2M}$, we get
\[
k\equiv u+aM\pmod{2M},\quad u\in\{1,2\},\ a\text{ odd}.
\]
By Lemma~\ref{lem:r2} (recall $M\ge4$) we may assume $r\not\equiv2\pmod{2M}$,
and hence $3\le r\le 2M-1$.

\medskip
Let \(\zeta\) be a primitive \(2M\)th root of unity. Since $P(r,2M,k,r-1,1)$ is perfect, Lemma~\ref{Primitive} shows that $F(\zeta)$ is a unit in \(\mathbb Z[\zeta]\), so by Lemma~\ref{OdoniCriterion} and Remark~\ref{epsil} there is an integer \(j\) with \(0\le j<2M\) such that
\[
F(\zeta)=\zeta^jF(\zeta^{-1}).
\]
Applying Lemma~\ref{upart}, we obtain
\begin{equation}\label{identity}
(1-\zeta)\bigl(\zeta^{r-1}-\zeta^{\,j-(r-1)}\bigr)
=
(1-\zeta^{r-1})
\bigl(\zeta^{\,j-(r-2)}-\zeta^{u-1}\bigr)
\bigl(1+\zeta^{-(u-1)}\bigr).
\end{equation}
Let $\alpha=\pi/(2M)$, so that $\zeta=e^{2i\alpha}$. Taking absolute values in \eqref{identity} and using
\[
\bigl|e^{2ia\alpha}-e^{2ib\alpha}\bigr|=2\bigl|\sin((a-b)\alpha)\bigr|,
\quad
\bigl|1+e^{-2i(u-1)\alpha}\bigr|=2\cos((u-1)\alpha),
\]
we obtain, in the notation of Lemma~\ref{noj},
\[
|1-\zeta|=2C,\quad
|1-\zeta^{r-1}|=2A,\quad
\bigl|\zeta^{r-1}-\zeta^{\,j-(r-1)}\bigr|=2D,
\]
\[
\bigl|\zeta^{\,j-(r-2)}-\zeta^{u-1}\bigr|
=2\bigl|\sin\bigl((j-(r-2)-(u-1))\alpha\bigr)\bigr|
=2\bigl|\sin\bigl((j-(r+u-3))\alpha\bigr)\bigr|=2B_u,
\]
and $\bigl|1+\zeta^{-(u-1)}\bigr|=2E_u$. Substituting into the modulus of \eqref{identity} gives
$(2C)(2D)=(2A)(2B_u)(2E_u)$, that is,
\begin{equation}\label{feqn}
2AB_uE_u=CD .
\end{equation}
Since $2M\ge8$ and $3\le r<2M$, Lemma~\ref{noj} asserts that \eqref{feqn} is impossible for $u\in\{1,2\}$. This contradiction shows that \(P(r,2M,k,r-1,1)\) is trivially perfect.
\end{proof}

The exponent in \eqref{identity} is $j-(r-2)$, as in Lemma~\ref{upart}; the shift to $j-(r+u-3)$ occurs only after taking the modulus, because $|\zeta^{\,j-(r-2)}-\zeta^{u-1}|$ depends on the \emph{difference} of the exponents.

\begin{cor}\label{lcor}
Let \(n=2^a3^b\), where \(a\ge0\) and \(b\in\{0,1\}\). Then
\(P(r,n,k,r-1,1)\) is perfect if and only if it is trivially perfect.
\end{cor}

\begin{proof}
Trivially perfect groups are perfect, so assume $P(r,n,k,r-1,1)$ is perfect; we must show it is trivially perfect.

\medskip
For every divisor $n'\mid n$ the group $P(r,n',k,r-1,1)$ is perfect. Indeed, the
$n'$th roots of unity are among the $n$th roots of unity, and the product of
$F(\zeta)$ over the remaining $n$th roots of unity is Galois-invariant, hence a
rational integer; so $R_{n'}(F)$ divides $R_n(F)=\pm1$.

\medskip
Applying Theorem~\ref{double} successively with $n=2^{a}3^{b},2^{a-1}3^{b},\ldots$, it is therefore sufficient to prove that $P(r,3^{b},k,r-1,1)$ is trivially perfect.

\medskip
If $b=0$ the odd part is $1$ and $P(r,1,k,r-1,1)$ is trivial, hence trivially perfect ($r\equiv0\pmod1$).

\medskip
If $b=1$, suppose $P(r,3,k,r-1,1)$ is not trivially perfect. Then $r\not\equiv0,1$ and $k\not\equiv1,2\pmod3$, so by Lemma~\ref{lem:reduce} we may take $r=2$ and $k=3$, giving $F(t)=1+t-t^{2}$ and
\[
R_3(F)=F(1)F(\rho)F(\rho^{2})=1\cdot\bigl(1+\rho-\rho^{2}\bigr)\bigl(1+\rho^{2}-\rho\bigr)
=\bigl|1+i\sqrt3\bigr|^{2}=4,
\]
where $\rho=e^{2\pi i/3}$. Hence $P(r,3,k,r-1,1)$ is not perfect, contradicting the perfectness of $P(r,3,k,r-1,1)$ established above. So $P(r,3,k,r-1,1)$ is trivially perfect, and Theorem~\ref{double} applied $a$ times gives the result.
\end{proof}

\section{Proof of main results}\label{sec:Main}

This section is devoted to proving the main results. We first establish a
reduction procedure and a symmetry principle that allow us to carry out all
arguments in a normalised setting (Section~\ref{normal}), and then transfer them
back to the general case. The main theorems are proved in
Section~\ref{sec:proofs}: Theorem~\ref{main}, Corollary~\ref{maincor},
Theorem~\ref{main2} and Corollary~\ref{main2cor}.

\medskip
Besides the results proved in Sections~\ref{sec:Odd} and~\ref{sec:Even}, the
arguments use the following results from the literature without proof:
\cite[Lemmas~2.1, 2.3, 2.6 and Theorem~C]{MR4315549},
\cite[Lemma~4.1 and Theorem~17(b)]{MR2946300}, and the free-product
decomposition of \cite{MR1961572}.

\subsection{Normalisation and arithmetic lifting}\label{normal}

Throughout this subsection, let
\[
d=\gcd(n,q),\quad n=dN,\quad q=dQ,
\]
so that $\gcd(N,Q)=1$ and $Q$ is invertible modulo $N$; fix $\widehat Q$ with $Q\widehat Q\equiv1\pmod N$. We choose $\widehat Q\in\{1,\dots,N\}$, so that $K\ge1$ (recall $K'\ge1$ because $d\mid k-1$ and $k\ge1$). All statements below depend on $K$ only modulo $N$
(Lemma~\ref{lem:reduce}).

By \cite[Lemma~4.1]{MR2946300}, the parameter transformation
\(
(r,k,s)\mapsto(s,\,n-k+2,\,r)
\)
induces an isomorphism
\[
P(r,n,k,s,q)\cong P(s,n,n-k+2,r,q),
\]
so all instances in which we have assumed $r\ge s$ are without loss of generality. Moreover, if $P(r,n,k,s,q)$ is perfect then $|r-s|=1$ by \cite[Lemma~2.1]{MR4315549}, and we restrict throughout to
\[
s=r-1 .
\]

If $P(r,n,k,s,q)$ is perfect then, by \cite[Lemma~2.3]{MR4315549}, $d\mid k-1$; write
\[
k-1=d(K'-1),\quad K=\widehat Q\,(K'-1)+1 .
\]
Under these hypotheses, the decomposition result of \cite{MR1961572} gives
\begin{equation}\label{decom}
P(r,n,k,r-1,q)\;\cong\;\ast^{\,d}\,P(r,N,K,r-1,1),
\end{equation}
a free product of $d$ copies. Since a free product is perfect if and only if each free factor is perfect, the general case reduces to the normalised case $P(r,N,K,r-1,1)$.

\begin{rem}
Note that $d\mid k-1$ is exactly the divisibility hypothesis added to Definition~\ref{def:tperf}; without it, neither \eqref{decom} nor Lemma~\ref{al} is available, which is the source of the counterexamples in Remark~\ref{rem:defnfix}.
\end{rem}

The following lemma records the relation between the arithmetic conditions for the original parameters and those for the reduced parameters.

\begin{lem}\label{al}
Assume $d=\gcd(n,q)$ divides $k-1$, and let $N,K$ be as above, with $s=r-1$. Then
\[
\begin{aligned}
K&\equiv1 \pmod N
&&\Longleftrightarrow&
k&\equiv1 \pmod n,\\
K&\equiv2 \pmod N
&&\Longleftrightarrow&
k&\equiv1+q \pmod n,\\
r&\equiv0 \pmod N
&&\Longleftrightarrow&
d\,r&\equiv0 \pmod n,\\
r&\equiv1 \pmod N
&&\Longleftrightarrow&
d\,s&\equiv0 \pmod n,\\
2(K-1)&\equiv1 \pmod N
&&\Longleftrightarrow&
2(k-1)&\equiv q \pmod n,\\
2r&\equiv1 \pmod N
&&\Longleftrightarrow&
q(r+s)&\equiv0 \pmod n,\\
\gcd(K-1-r,\,N)&=1
&&\Longleftrightarrow&
\gcd(k-1-qr,\,n)&=d.
\end{aligned}
\]
In particular, $P(r,n,k,r-1,q)$ is trivially perfect if and only if $P(r,N,K,r-1,1)$ is, and it is of type $\widetilde{\mathfrak Z}$ if and only if $P(r,N,K,r-1,1)$ is.
\end{lem}

\begin{proof}
The third and fourth equivalences are immediate from $n=dN$ and $s=r-1$.

\medskip
For the first two and the fifth, note that $k-1=d(K'-1)$ and $K-1=\widehat Q(K'-1)$, so that $K'-1\equiv Q(K-1)\pmod N$. Hence, for any integer $\lambda$,
\[
\begin{aligned}
k-1&\equiv \lambda q\pmod n
\iff d(K'-1)\equiv \lambda dQ\pmod{dN}\\
&\iff Q(K-1)\equiv\lambda Q\pmod N
\iff K-1\equiv\lambda\pmod N.
\end{aligned}
\]
using $\gcd(Q,N)=1$. Taking $\lambda=0,1$ gives the first two equivalences. For the fifth, $2(k-1)\equiv q\pmod n$ is $2d(K'-1)\equiv dQ\pmod{dN}$, i.e. $2Q(K-1)\equiv Q\pmod N$, i.e. $2(K-1)\equiv1\pmod N$.

\medskip
For the sixth, with $s=r-1$ we have $q(r+s)=q(2r-1)$, and
\[
dQ(2r-1)\equiv0\pmod{dN}\iff Q(2r-1)\equiv0\pmod N\iff 2r\equiv1\pmod N .
\]
For the last equivalence, $k-1-qr=d\bigl((K'-1)-Qr\bigr)$, and since multiplication by the unit $\widehat Q$ modulo $N$ does not change the greatest common divisor with $N$,
\[
\gcd(k-1-qr,\,n)=d\,\gcd\bigl((K'-1)-Qr,\,N\bigr)=d\,\gcd(K-1-r,\,N).
\]

The final assertions follow by comparing the list with Definitions~\ref{def:tperf} and~\ref{def:typeZtilde}.
\end{proof}

\begin{prop}\label{tperfect}
Assume $d=\gcd(n,q)$ divides $k-1$ and $s=r-1$. If $d\,r\equiv0\pmod n$ or $d\,s\equiv0\pmod n$, then $P(r,n,k,s,q)$ is perfect.
\end{prop}

\begin{proof}
By Lemma~\ref{al} the hypotheses are equivalent to $r\equiv0$ or $r\equiv1\pmod N$. By \cite[Lemma~2.6]{MR4315549} the involution $(K,r)\mapsto(r+1,K-1)$ preserves perfectness, and by \cite[Theorem~17(b)]{MR2946300}, if $K\equiv1,2\pmod N$ then $P(r,N,K,r-1,1)$ is trivial, hence perfect. Applying the involution, $P(r,N,K,r-1,1)$ is perfect whenever $r\equiv0,1\pmod N$. Finally, by \eqref{decom}, $P(r,n,k,s,q)$ is a free product of $d$ copies of $P(r,N,K,r-1,1)$ and is therefore perfect.
\end{proof}

\begin{cor}\label{cor:tpisp}
Every trivially perfect group is perfect.
\end{cor}

\begin{proof}
Let $P(r,n,k,s,q)$ be trivially perfect; by Definition~\ref{def:tperf}, $|r-s|=1$ and $d\mid k-1$, and after the symmetry $(r,k,s)\mapsto(s,n-k+2,r)$ we may assume $s=r-1$. If $k\equiv1$ or $k\equiv1+q\pmod n$ then the group is trivial by \cite[Theorem~17(b)]{MR2946300}. Otherwise $dr\equiv0$ or $ds\equiv0\pmod n$, and Proposition~\ref{tperfect} applies.
\end{proof}

\subsection{Proofs}

Throughout, suppose $P(r,n,k,s,q)$ is perfect. Then $|r-s|=1$ by \cite[Lemma~2.1]{MR4315549}, and since $r\ge s$ we get $s=r-1$; moreover $d\mid k-1$ by \cite[Lemma~2.3]{MR4315549}. Hence the normalisation of Section~\ref{normal} applies, and by \eqref{decom} perfectness of $P(r,n,k,s,q)$ is equivalent to perfectness of the normalised group $P(r,N,K,r-1,1)$.

\begin{proof}[Proof of Theorem~\ref{main}]
Suppose $P(r,n,k,s,q)$ is perfect but not trivially perfect. By the previous paragraph, $s=r-1$ and $P(r,N,K,r-1,1)$ is perfect, and by Lemma~\ref{al} it is not trivially perfect either; Hence,
\[
r\not\equiv0,1 \pmod N,\quad K\not\equiv1,2\pmod N .
\]
In particular $N>1$, and $N\mid n$ gives $\gcd(N,6)=1$.

\medskip
Let \(\zeta\) be a primitive \(N\)th root of unity, and let $F(t)$ be the polynomial \eqref{eq2} with $k$ replaced by $K$ and with $r$ reduced modulo $N$ (Lemma~\ref{lem:reduce}); note $r\ge2$ after reduction, since $r\not\equiv0,1$. Since \(P(r,N,K,r-1,1)\) is perfect, Lemma~\ref{Primitive} implies that \(F(\zeta)\) is a unit in \(\mathbb{Z}[\zeta]\). Hence by Lemma~\ref{OdoniCriterion} there exist an integer \(j\) with \(0\le j<N\) and \(\epsilon\in\{\pm1\}\) such that
\[
\epsilon\zeta^jF(\zeta)=F(\zeta^{-1}).
\]
Lemma~\ref{mainlem} now gives $2r\equiv1$ or $2K\equiv3\pmod N$; by Definition~\ref{def:typeZtilde} these say precisely that \(P(r,N,K,r-1,1)\) is of type \(\mathfrak{Z}'\) or of type \(\mathfrak{Z}\) respectively (for the latter, $2(K-1)\equiv1$). Hence \(P(r,N,K,r-1,1)\) is of type \(\widetilde{\mathfrak{Z}}\), and Lemma~\ref{al} shows that \(P(r,n,k,s,q)\) is of type \(\widetilde{\mathfrak{Z}}\) as well.
\end{proof}

\begin{proof}[Proof of Corollary~\ref{maincor}]
Let \(G_n(w)\) and \(G_n(w')\) be as in the statement, with the same exponent sum for each generator. Then \(G_n(w)\) is perfect if and only if \(G_n(w')\) is perfect, so it suffices to determine the perfectness of
\(
P(r,n,k,s,q)=G_n(w').
\)

\medskip
Assume first that \(G_n(w')\) is perfect. By \cite[Lemma~2.1]{MR4315549} we have \(|r-s|=1\), and since \(r\ge s\) it follows that \(s=r-1\). By Theorem~\ref{main}, either \(G_n(w')\) is trivially perfect, which is case (1), or it is of type \(\widetilde{\mathfrak{Z}}\). In the latter case, \(d\mid k-1\) by \cite[Lemma~2.3]{MR4315549}, so \eqref{decom} applies and $P(r,N,K,r-1,1)$ is perfect, and by Lemma~\ref{al} it is of type $\widetilde{\mathfrak Z}$. If $N=1$, then $\gcd(K-1-r,N)=1$ holds trivially. If $N\ge2$, then $\gcd(N,K-1,1)=1$, so Theorem~\ref{ThmC} applies to $P(r,N,K,r-1,1)$ and gives $\gcd(K-1-r,N)=1$. In either case Lemma~\ref{al} then yields
\[
\gcd(k-1-qr,\,n)=d=\gcd(n,q),
\]
which is case (2).

\medskip
Conversely, suppose $s=r-1$. In case (1) the group is perfect by Corollary~\ref{cor:tpisp}. In case (2), $\gcd(k-1-qr,n)=d$ forces $d\mid k-1$, so \eqref{decom} applies; by Lemma~\ref{al} the normalised group $P(r,N,K,r-1,1)$ is of type $\widetilde{\mathfrak Z}$ with $\gcd(K-1-r,N)=1$, so Theorem~\ref{ThmC} shows it is perfect, and hence so is the free product \eqref{decom}.
\end{proof}

\begin{proof}[Proof of Theorem~\ref{main2}]
Let $n=2^{m}M$ with $\gcd(M,6)=1$, and suppose $P(r,n,k,s,q)$ is perfect and \emph{not} trivially perfect; we must show that $P(r,M,k,s,q)$ is of type $\widetilde{\mathfrak Z}$.

\medskip
As above, $s=r-1$, $d=\gcd(n,q)$ divides $k-1$, and by \eqref{decom} the normalised group $P(r,N,K,r-1,1)$ is perfect, where $N=n/d$. By Lemma~\ref{al} it is not trivially perfect.

\medskip
Write $N=2^{m'}M'$ with $M'$ odd. Since $M'\mid N\mid n$ and $M'$ is odd, we have $M'\mid M$, so $\gcd(M',6)=1$. Moreover, as $\gcd(M,2)=1$,
\begin{equation}\label{eq:oddpart}
M'=\frac{M}{\gcd(M,q)},
\end{equation}
because the odd part of $d=\gcd(2^{m}M,q)$ is $\gcd(M,q)$.

\medskip
We now apply Theorem~\ref{double} repeatedly within the normalised family, where the second and fifth parameters are fixed at $r-1$ and $1$, so that no further normalisation is required at intermediate levels. Suppose, for contradiction, that $P(r,M',K,r-1,1)$ is trivially perfect. Since $P(r,N,K,r-1,1)$ is perfect, so is $P(r,2^{i}M',K,r-1,1)$ for every $0\le i\le m'$ (the $2^{i}M'$th roots of unity form a subset of the $N$th roots of unity, and the product of $F$ over the remaining $N$th roots of unity is Galois-invariant, hence a rational integer, so $R_{2^{i}M'}(F)$ divides $R_{N}(F)=\pm1$). Applying Theorem~\ref{double} successively with $n=2M'$, $4M'$, \ldots, $2^{m'}M'$ therefore yields that $P(r,N,K,r-1,1)$ is trivially perfect, a contradiction.

\medskip
Hence $P(r,M',K,r-1,1)$ is perfect but not trivially perfect. Since $\gcd(M',6)=1$, Theorem~\ref{main} shows that it is of type $\widetilde{\mathfrak Z}$.

\medskip
Finally, write $d'=\gcd(M,q)$, so $M'=M/d'$ by \eqref{eq:oddpart}, and
$d=2^{e}d'$ for some $e\ge0$. We have $d'\mid d\mid k-1$, so Lemma~\ref{al} may be
applied with $n$ replaced by $M$; the normalised parameters of
$P(r,M,k,r-1,q)$ are $M'$ and $K^{\ast}$, where $K^{\ast}-1=\widehat Q^{\ast}\,(k-1)/d'$,
$Q^{\ast}=q/d'$ and $Q^{\ast}\widehat Q^{\ast}\equiv1\pmod{M'}$. Now
$Q^{\ast}=2^{e}Q$, so we may take $\widehat Q^{\ast}\equiv2^{-e}\widehat Q\pmod{M'}$
(note $M'\mid N$, so $\widehat Q$ is also an inverse of $Q$ modulo $M'$, and $2$ is
invertible modulo the odd number $M'$). Also $(k-1)/d'=2^{e}(K'-1)$. Hence
\[
K^{\ast}-1\equiv2^{-e}\widehat Q\cdot2^{e}(K'-1)=\widehat Q(K'-1)=K-1\pmod{M'},
\]
that is, $K^{\ast}\equiv K\pmod{M'}$. Hence, $P(r,M',K^{\ast},r-1,1)$ and
$P(r,M',K,r-1,1)$ have the same type conditions, so the latter being of type
$\widetilde{\mathfrak Z}$ and Lemma~\ref{al} (with $n$ replaced by $M$) show that
$P(r,M,k,s,q)$ is of type $\widetilde{\mathfrak Z}$.
\end{proof}

The induction above is carried out after normalising to $q=1$. This is necessary: $\gcd(2^{i}M,q)$ varies with $i$ (for example, $\gcd(M,2)=1$ but $\gcd(2M,2)=2$), so the normalised modulus attached to $P(r,2^{i}M,k,r-1,q)$ is not obtained from that attached to $P(r,2^{i-1}M,k,r-1,q)$ simply by doubling. Within the family with $q=1$ no such difficulty arises.

\begin{proof}[Proof of Corollary~\ref{main2cor}]
Trivially perfect groups are perfect by Corollary~\ref{cor:tpisp}. Conversely, suppose $P(r,n,k,s,q)$ is perfect with $n=2^{a}3^{b}$, $b\in\{0,1\}$. Then $s=r-1$ and $d\mid k-1$, so by \eqref{decom} the group $P(r,N,K,r-1,1)$ is perfect, where $N=n/d$. Since $N\mid n$, we have $N=2^{a'}3^{b'}$ with $b'\in\{0,1\}$, so Corollary~\ref{lcor} shows that $P(r,N,K,r-1,1)$ is trivially perfect. By Lemma~\ref{al}, $P(r,n,k,s,q)$ is trivially perfect.
\end{proof}

\section*{Acknowledgements}
The first author acknowledges financial support from Southern Illinois University and the Simons Foundation during the preparation of this article. The second author was supported by the University of Pretoria Research Development Programme during the completion of this work.

\bibliographystyle{plain}
\bibliography{Refs}

\end{document}